\documentclass{article}
\usepackage[top=25mm, bottom=25mm, left=25mm, right=25mm]{geometry}
\usepackage{lmodern}
\usepackage[utf8]{inputenc}
\usepackage{default}
\usepackage[english]{babel}
\usepackage{calc}
\usepackage{multicol}
\usepackage{mathtools}
\usepackage{amsmath,amsfonts,amsthm,amssymb}
\usepackage[T1]{fontenc}
\usepackage{graphicx}
\usepackage{xcolor}
\usepackage{tikzsymbols}
\usepackage{csquotes}
\usepackage{varwidth}
\usepackage{array}
\usepackage{graphicx}
\usepackage{bm}
\usepackage{listings}
\usepackage{authblk}

\newcommand{\mbb}{\mathbb}
\newcommand{\mc}{\mathcal}
\newcommand{\tn}{\textnormal}

\newcommand{\tsc}{\textsc}
\newcommand{\ttt}{\texttt}
\newcommand{\mbf}{\mathbf}

\newtheorem{remark}{Remark}

\newtheorem{proposition}{Proposition}

\providecommand{\keywords}[1]
{
  \small	
  \textbf{\textit{Keywords---}} #1
}

\begin{document}
\title{\textbf{The Fast and Free Memory Method for the efficient computation of convolution kernels}}
\date{\today}
\author{
    Matthieu~Aussal, CMAP, CNRS, Ecole polytechnique, Institut Polytechnique de Paris, 91128 Palaiseau, France, \texttt{matthieu.aussal@polytechnique.edu}\\\vspace{0.5\baselineskip}
    Marc~Bakry, CMAP, CNRS, Ecole polytechnique, Institut Polytechnique de Paris, 91128 Palaiseau, France, \texttt{marc.bakry@polytechnique.edu}
}

\maketitle

\begin{abstract}
    We introduce the Fast Free Memory method (FFM), a new fast method for the numerical evaluation of convolution products. Inheriting from the Fast Multipole Method, the FFM is a descent-only and kernel-independent algorithm. We give the complete algorithm and the relevant complexity analysis. While dense matrices arise normally in such computations, the linear storage complexity and the quasi-linear computational complexity enable the evaluation of convolution products featuring up to one billion entries. We show how we are able to solve complex scattering problems using Boundary Integral Equations with dozen of millions of unknowns. Our implementation is made freely available within the \tsc{Gypsilab} framework under the GPL 3.0 license.
\end{abstract}


\section{Introduction}\label{sec:introduction}

The numerical computation of convolution products is a crucial issue arising in many domains like the filtering, the computation of boundary integral operators, optimal control, etc. In a continuous framework, a convolution product is of the form

\begin{align}\label{eq:convolution_continuous}
    v(x) = \int_{\Omega}{G(x, y)\, u(y)\,d\Omega_y}
\end{align}
where $\Omega$ is some domain of integration in $\mbb R^d,d\in\mbb N^{\star}$, $u$ some function. The bivariate function $G(.,.)$ is some {\em convolution kernel} of the form 

\begin{align}
    G(x,y) = G(x - y)
\end{align}
where $\|\,.\,\|$ is some distance which can be the classical euclidean distance $|\,.\,|$. Of course, eq. (\ref{eq:convolution_continuous}) does not admit an analytical expression in the general case and the integral is computed numerically using, for instance, a quadrature rule. Assuming we want to evaluate $v$ on a finite set of nodes $X = (x_i)_{i\in[\![1,N_X]\!]}$, we have

\begin{align}\label{eq:convolution_discrete}
    v(x_i) \approx \sum_{j = 1}^{N_Y}{\omega_j\,G(x_i, y_j)\,u(y_j)}
\end{align}
where $(\omega_j)_{j\in[\![1,N_Y]\!]}$ and $Y = (y_j)_{j\in[\![1,N_Y]\!]}$ are respectively the weights and nodes of such a quadrature. The {\em discrete convolution} may be recast as a simple matrix-vector product 

\begin{align}\label{eq:convolution_matrix}
    \mbf v = \mbf G\cdot \mbf W \cdot \mbf u
\end{align}
where $\mbf u = (u_j)_j = (u(y_j))_j$, $\mbf v = (v_i)_i = (v(x_i))_i$, $\mbf G = G(x_i,y_j))_{\{i,j\}}$ and $\mbf W= \tn{diag}((\omega_j)_j)$. Obviously, the matrix $\mbf G$ is {\em dense}. Therefore, the storage and computational cost grow quadratically. The computation of (\ref{eq:convolution_matrix}) is constrained to smaller problems ($N_X, N_Y\approx$ a few thousand) on personal computers and smaller servers, and to $N_X,N_Y\approx 10^6$ for industrial servers.\par\vspace{\baselineskip}
The current approach is to perform the computation approximately up to a given tolerance $\varepsilon$ (accuracy). In the past thirty years, multiple so-called {\em acceleration} methods have been proposed. The entries of $\mbf G$ can be seen as the {\em description of an interaction} between a {\em source} set of nodes $Y$ and a {\em target} set of nodes $X$. Thus all blocks of $\mbf G$ describe the interaction between a {\em source} subset of $Y$ and a {\em target} subset of $X$. Theses interactions may be {\em compressible}, i.e. it admits a {\em low-rank} representation. This is the case, for example, when two subsets are far enough following an {\em admissibility criterion}. The methods mentioned in the following propose different alternatives on the way the interactions are characterized and computed. The standard way is probably the {\em Fast Multipole Method} (FMM) developed by L. Greengard and V. Rokhlin (see \cite{greengardFMM}), initially introduced for the computation of the gravitational potential of a cloud of particles. Later versions feature the support of oscillatory kernels like the Helmholtz Green kernel. One major drawback is that the implementations are mostly kernel-specific despite recent advances in the domain \cite{fong2009}. We refer to \cite{chengFMM} for more details. In 1999, a new approach named {\em Hierarchical matrices} ($\mc H$-matrices) was introduced by S. B\"orm, L. Grasedyck and W. Hackbusch. This method is based on the representation of the matrix by a quadtree whose leaves are low-rank or full-rank submatrices. A strong advantage in favor of hierarchical matrices is that a complete algebra has been created: addition, multiplication, LU-decomposition, etc. Unfortunately, $\mc H$-matrices become less effective for strongly oscillating kernels because the rank of the compressible blocks increases with the frequency of the oscillations. For more details and a complete mathematical analysis, we refer to \cite{hackbusch2015,borm2015}. One of the most recent compression methods, to our knowledge, is the Sparse Cardinal Sine Decomposition (SCSD) proposed by F. Alouges and M. Aussal in 2015 \cite{alouges2015}. It is based on a representation of the Green kernel in the Fourier domain using the integral representation of the cardinal sine function. One major advantage is that the matrix-vector product is performed without partitioning of the space. However, there is no corresponding algebra. All the aforementioned methods aim at having a quasi-linear storage and computational complexity.\par\vspace{\baselineskip}
We introduce here the Fast Free Memory method (FFM) which blends together multiple features of the existing methods in order to have a minimalist storage requirement. It is designed for the computation of massive matrix-vector products, up to billions of nodes, where methods featuring quasi-linear storage complexity fail because of the non-linear part. The FFM relies heavily on the FMM algorithm, in particular on the octree-based space partitioning, and introduces compression techniques featured in the $\mc H$-matrices and the SCSD. In particular, the FFM is a kernel-independent method for the computation of the matrix-vector product (\ref{eq:convolution_discrete}). A powerful feature is that the required storage is minimalist with {\em linear} complexity and that the computational complexity is quasi-linear. This enables the computation of matrix-vector products with hundred of millions of nodes on the {\em source} and {\em target} sets with laboratory-sized servers. This paper is divided into three main parts. In the first part, we develop the FFM algorithm for standard kernels (non-oscillating kernels and the Helmholtz Green kernel). In the second part, we prove the storage and computational complexities. In the last part, we illustrate the proven complexities on academic and industrially-sized problems with millions of entries up to one billion.
\section{The FFM algorithm}

The FFM is a divide-and-conquer method recursively implemented. It is based on two partitioning trees (one for the {\em source} set and one for the {\em target} set) strongly related to the octree used in the classical FFM. This inheritance and the particularities of the FFM are explained in the first part of this section where we also introduce the notations which will be used in this paper. In the second part, we develop the different steps of the algorithm for the case of a general convolution and we optimize it for the case of oscillatory kernels.

\subsection{The FFM and the FMM}

The basic idea in the FMM is that the interaction between subsets of nodes {\em sufficiently far} one from another admit a low-rank representation. The space is therefore partitioned using two octrees (see Figure \ref{img:octree}) obtained by successive refinements of the bounding boxes of the initial {\em source} and {\em target} set of nodes. At each level of refinement, boxes far one from the other (following a given criterion) correspond to {\em compressible} interactions. The other boxes are further subdivided and yield in turn low-rank and non-compressible interactions. When the matrices corresponding to the non-compressible interactions are small-enough, a full computation is performed. For more details, we refer to the bibliography on the topic, for example \cite{greengardFMM, chengFMM, darve1999, darve2000}. In the FMM, the low-rank approximation is performed using a multipole expansion of the convolution kernel, see \cite{greengardHuang}. As this will be illustrated in the following section \ref{subsec:ffmAlgo}, the FFM is more versatile and such expansions will be used only for the oscillating kernels.\par\vspace{\baselineskip}
\begin{figure}[!htbp]
    \centering
    \includegraphics[width=0.8\linewidth]{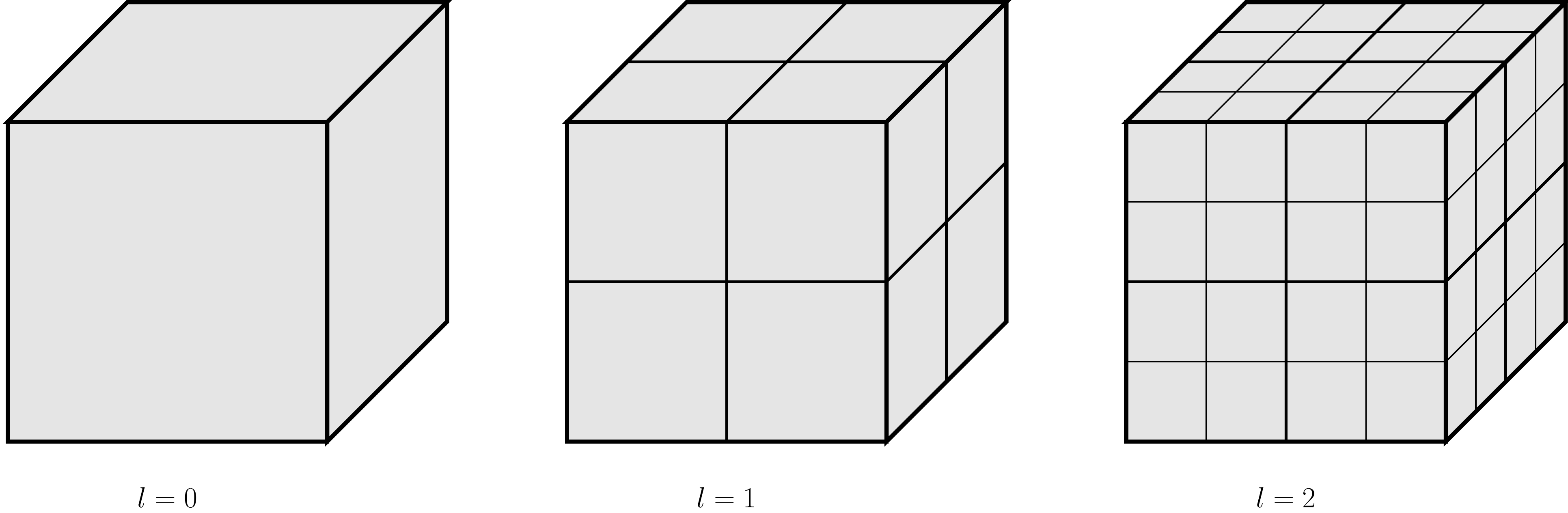}
    \caption{Representation of the octree used in the FFM with three levels of refinement.}
    \label{img:octree}
\end{figure}
The differences between the two methods find their origin in the way the partitioning octrees are built. Classically, the root bounding-box for each set of nodes is {\em tight} in the sense that its edge dimensions fit tightly the expansion of the set of nodes in each direction. Consequently, the initial bounding boxes for the {\em source} and {\em target} sets have different sizes and shapes in general. In the FFM, both initial bounding boxes are cubic with the {\em same} edge length. In particular, it means that for a given level of refinement all the boxes in both octrees are the same up to a translation. The direct consequence of this choice is that there is no need to interpolate data between the boxes of the different trees because the discretization in the quadrature spaces are the same. Another difference is that the FFM is a {\em descent-only} algorithm. This last particularity enables the {\em linear} storage complexity, see section \ref{sec:complexity}.

\subsection{The FFM algorithm}\label{subsec:ffmAlgo}
In this subsection, we describe the kernel-independent compression method for the matrix-vector product. We first describe the initialization. Then, we explain how the kernel-independent matrix-vector product is computed using a well-known Lagrange-interpolation method. We show how the user can choose the compression method when dealing with the particular case of the Helmholtz Green kernel. Finally, a stopping criterion is proposed. In the following, we reuse the previous notations and we introduce $l$ as the depth of the octree. The case $l=0$ corresponds to the root and $l=l^{\tn{max}}$ is the maximum allowed depth.

\subsubsection{Initialization}
The initialization ($l=0$) is performed easily by computing the bounding boxes for $X$ and $Y$. Let $d_{X,\tn{max}}$ and $d_{Y,\tn{max}}$ be the maximum edge dimension of $X$, resp. $Y$, then the initial dimension for {\em both} bounding boxes is 

\begin{align}
    d_0 = \max(d_{X,\tn{max}},d_{Y,\tn{max}}).
\end{align}
The initial bounding box for each set is simply a cube enclosing $X$, resp. $Y$, with edge length $d_0$. At the depth $l$ in the octree, the dimension of the bounding boxes is simply $d_l = d_0/2^l$.

\subsubsection{The kernel-independent FFM }\label{subsubsec:generic_ffm_iteration}
We suppose that the current refinement level is $l\in[\![1,l^{\tn{max}}]\!]$. We consider only one interaction between a box of the {\em source} tree and a box of the {\em target} tree. This interaction is low-rank if the distance between their centers is greater than two-times the edge length of a box. Let $m$ the number of nodes in the {\em target} box and $n$ the number of nodes in the {\em source} box, the compressed product is performed using a classical bivariate Lagrange interpolation, see for example \cite{cambier2017} or the chapters related to the $\mc H^2$-matrices in \cite{hackbusch2015}. The principle of the Lagrange interpolation is to approximate the convolution kernel like

\begin{align}
    G(x,y) \approx \sum_{i=1}^{r_X}{\mc L_i(x)\sum_{j=1}^{r_Y}{G(x_{c,i},y_{c,j})\,\mc L_j(y)}},
\end{align}
where
\begin{itemize}
    \item $\{x_{c,i}\}_i$ and $\{y_{c,j}\}_j$ are the {\em target} and {\em source} {\em control} nodes for the Lagrange polynomials. Following \cite{mastroianni2001, cambier2017}, the best interpolation nodes are the Chebyshev nodes.
    \item $\mc L_i(x)$, resp. $\mc L_j(y)$ are the Lagrange polynomials defined as the tensorization of the one-dimensional Lagrange polynomials in each direction of the space and localized at the {\em control} nodes.
    \item $r_X$ and $r_Y$ are the ranks of the interpolation for the target and source variables. If $r_{X,d}, r_{Y,d}$ are the rank of the interpolation in {\em the direction $d$}, then $r_X = \prod_{d=1}^3{r_{X,d}}$ and $r_Y = \prod_{d=1}^3{r_{Y,d}}$. For a prescribed accuracy $\varepsilon$ on the interpolation, these ranks depend solely on the size of the interpolation domain (in other words: the size of the bounding boxes) and on the regularity of the kernel being interpolated. Consequently, we have 

    \begin{align}
        r_X\equiv r_Y \equiv r \equiv (r_1)^3
    \end{align}
    in the FFM framework.
\end{itemize}
The interpolated matrix-vector product can be therefore recast as

\begin{align}
    \mbf v \approx \mbf L^T_X\cdot\left(\mbf T\cdot\left(\mbf L_Y\cdot \mbf u\right)\right)
\end{align}
where 
\begin{itemize}
    \item $\mbf u$ is the {\em source} vector whose entries are localized at the nodes contained within the {\em source} box.
    \item $\mbf L_Y: r\times n$, resp. $\mbf L_X: r\times m$, is the {\em interpolation} matrix whose entries are the Lagrange polynomials localized at the {\em source} control nodes evaluated at the {\em source} nodes.
    \item $\mbf T:r\times r$ is the {\em transfer}-matrix whose entries are the kernel evaluated for each possible couple of {\em target} and {\em source} control node.
\end{itemize}
However, $r$ is not necessarily low and $\mbf T$ can be further compressed using the Adaptive Cross Approximation, see \cite{bebendorf}, such that 

\begin{align}
    \mbf T\approx \mbf A\cdot\mbf B^T
\end{align}
where $\mbf A: r\times r_{\mbf T}$ and $\mbf B:r\times r_{\mbf T}$ such that $r_{\mbf T}\ll r$. Finally, the Lagrange interpolation consists in four successive matrix-vector products such that

\begin{align}\label{eq:ffmACAmatvec}
    \mbf v \approx \mbf L^T_X\cdot\left(\mbf A\cdot\left(\mbf B^T\cdot\left(\mbf L_Y\cdot \mbf u\right)\right)\right).
\end{align}\par
As the rank $r$ depends on the kernel, it may increase unacceptably when dealing with oscillatory kernels because the polynomial order must be high to fit the oscillations in the sub-domains. In that case, it is beneficial to use kernel-specific compression method as illustrated in the next section \ref{subsubsec:oscillating_ffm_iteration}.

\subsubsection{The FFM for oscillatory kernels}\label{subsubsec:oscillating_ffm_iteration}
In this section, we illustrate the change of compression method for the special case of the Helmholtz Green kernel

\begin{align}\label{eq:helmholtzKernel}
    G(x,y) =\dfrac{1}{4\pi} \dfrac{e^{i k |x-y|}}{|x-y|}
\end{align}
where $k$ is the {\em wavenumber}. A test is performed on the value of $k\cdot d_l$ to determine wether the low-rank interaction is oscillating. For example, the evaluation of the Helmholtz kernel (\ref{eq:helmholtzKernel}) for two nodes $x_i$ and $y_j$ sufficiently close can be detected as non-oscillating because the value of $k\cdot |x_i - y_j|$ is small. In this case, we use the Lagrange interpolation described in subsubsection \ref{subsubsec:generic_ffm_iteration} instead of a specific low-frequency FMM (see \cite{greengard1998} or \cite{darve2004}). In the other case, we approximate the kernel using its Gegenbauer-series expansion like in the FMM.\par 
Let $x_0$, resp. $y_0$, be the center of the {\em target}, resp. {\em source}, box, then

\begin{align}
    x - y = (x-x_0) + (x_0 - y_0) + (y_0 - y)
\end{align}
which can be reformulated like

\begin{align}
    \mbf r = \mbf r_0 + \mbf r_{xy}
\end{align}
where $\mbf r_0 = x_0 - y_0$. Let also $\mbb S^2$ be the unit sphere in $\mbb R^3$, then following for example \cite{darve1999} 

\begin{align}\label{eq:gegenbauerIntegral}
    \dfrac{e^{ik r}}{r} = ik\lim_{L\to\infty}{ \int_{\mbb S^2}{ e^{ik\hat s\cdot\mbf r_{xy}} T_{L,\mbf r_0}(\hat s)\,d\hat s } }
\end{align}
where $r=|x-y|$ and $T_{L,\mbf r_0}(\hat s)$ is the Gegenbauer series such that

\begin{align}\label{eq:gegenbauerSerie}
    T_{L,\mbf r_0}(\hat s) = \sum_{p=1}^L{\dfrac{(2p+1)i^p}{4\pi} h_p^{(1)}(k\, r_0) P_p(\hat s\cdot\hat{\mbf r}_0)},\hspace{5mm}\hat{\mbf r}_0 = \mbf r_0/|\mbf r_0|,\hspace{5mm} r_0 = |\mbf r_0|
\end{align}
where $h_p^{(1)}$ is the spherical Hankel function of the first kind and of order $p$, $P_p$ is the Legendre polynomial of order $p$. In practice, (\ref{eq:gegenbauerSerie}) is truncated at the rank 

\begin{align}\label{eq:gegenbauerRank}
    L = \lfloor |k|\cdot\sqrt 3\cdot d_l -\log(\varepsilon)\rfloor
\end{align}
where $d_l$ is the size of the edge of the bounding box and $\varepsilon$ is the prescribed accuracy on the matrix-vector product, see \cite{darve2000}. The integral (\ref{eq:gegenbauerIntegral}) is computed using a spherical quadrature $\left\{\omega_q,\hat s_q\right\}_{q\in[\![1,n_Q]\!]}$ such that, for two nodes $x_i$ and $y_j$,

\begin{align}\label{eq:gegenbauerQuadrature}
    \dfrac{e^{ik |x_i - y_j|}}{|x_i - y_j|} \approx ik\sum_{q=1}^{n_Q}{e^{ik x_i\cdot\hat s_q}\left(\left(\omega_q T_{L,\mbf r_0}(\hat s_q)e^{ik\mbf r_0\cdot\hat s_q}\right) e^{-ik\hat s_q\cdot y_j}\right)}.
\end{align}
The computation is therefore recast as three successive matrix vector products

\begin{align}
    \mbf v \approx \mbf F_{\hat s\to X}\cdot\left(\mbf T_{L,\mbf r_0}\cdot \left(\mbf F_{Y\to\hat s}\cdot \mbf u\right)\right)
\end{align}
where
\begin{itemize}
    \item $\mbf F_{Y\to\hat s}$ is the {\em dense} matrix representing the discrete non-uniform forward Fourier transform from the {\em source} set to the Fourier domain such that 

    \begin{align}
        \tilde u_q = \sum_{j=1}^{n}{e^{-ik\hat s_q\cdot y_j} u_j}.
    \end{align}
    \item $\mbf T_{L,\mbf r_0}$ is the {\em transfer} {\em diagonal} matrix whose entries contain the value of Gegenbauer series evaluated at $\hat s_q$ such that 
  
    \begin{align}
        \tilde v_q = ik\,\omega_q \left(T_{L,\mbf r_0}(\hat s_q)e^{ik\mbf r_0\cdot\hat s_q}\right)\tilde u_q.
    \end{align}
    \item $\mbf F_{\hat s\to X}$ is the {\em dense} matrix representing the discrete non-uniform backward Fourier transform from the Fourier domain to the {\em target} set such that 

    \begin{align}
        v_i = \sum_{q=1}^{n_Q}{e^{ik x_i\cdot\hat s_q}\tilde v_q}.
    \end{align}
\end{itemize}
Of course, the Fourier-matrices are {\em never} assembled and the Fourier transforms are computed using the corresponding Non-Uniform Fast Fourier Transform (NUFFT) introduced  by A. Dutt and V. Rokhlin \cite{duttNufft1993}, later improved by Greengard \cite{greengardNufft}. We refer to these papers for a complete analysis of the algorithm.
\begin{remark}
    In the classical FMM, the Fast Fourier Transforms are computed at the deepest level in the trees, forcing a back-propagation in the octree. In the FFM, they are performed on-the-fly enabling a descent-only algorithm.
\end{remark}

\subsubsection{Stopping criterion}
The recursive partitioning is stopped whenever one of the following is verified:
\begin{itemize}
    \item any compressible interaction has been computed,
    \item the average number of nodes in the boxes is below some value.
\end{itemize}
The remaining non-compressible interactions, if any, are computed as a dense matrix-vector product.
\section{Complexity analysis}\label{sec:complexity}
In this section, we prove the $\mc O(N)$ storage complexity and the $\mc O(N\cdot\log(N))$ computational complexity of the FFM, where $N=\max(N_X,N_Y)$. To that purpose, we assume that each set $X$ or $Y$ consists in an uniform distribution of nodes in a cube. The case of surface node-distributions is eventually tackled as a particular case.

\subsection{Complexity of an octree}
We give here some general well-known results on space partitioning trees. Assuming $d_0$ is the length of the edge of the root bounding box, the bounding boxes at level $l$ have the dimension $d_l = d_0/2^l$. They contain (in average) $n = N/8^l$ nodes. There are {\em at most} $8^l$ non-empty boxes. Consequently, the maximum depth of an octree is $l^{\tn{max}} = \lfloor \log_8(N)\rfloor$. The construction of the tree itself requires $\mc O(N\cdot\log_8(N))$ operations. In general, the required storage is also $\mc O(N\cdot\log_8(N))$ but it is $\mc O(N)$ in the FFM framework because we store only the data needed at the current depth. For the sake of simplicity, we assume that the boxes contain {\em exactly} $N/8^l$ nodes as it does not modify the overall estimate.

\paragraph{The particular case of surface distributions} We emphasize the fact that using an octree to subdivide evenly distributed nodes on a surface amounts to consider a plane surface partitioned using a quadtree as illustrated on Figure \ref{img:surfaceQuadtree}.
\begin{figure}[!htbp]
    \centering
    \includegraphics[height=60mm]{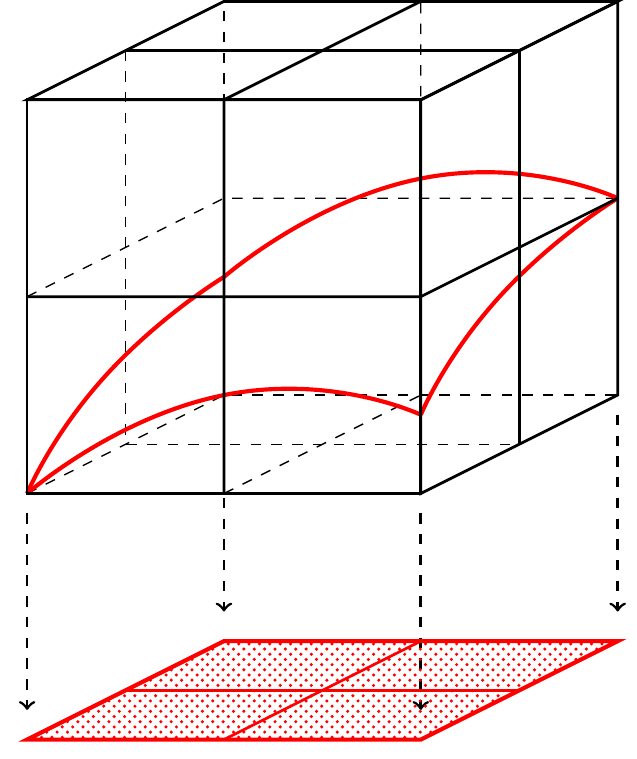}
    \caption{Illustration of the subdivision of a surface using an octree.}
    \label{img:surfaceQuadtree}
\end{figure}
Consequently, we replace the $8$ by $4$ in the aforementioned results: the maximum depth is $\lfloor \log_4(N)\rfloor$, etc.
\begin{remark}
    For the sake of simplicity, the subscript for the $\log$ indicating the type of the logarithm is omitted whenever it is not required for the comprehension.
\end{remark}

\subsection{Complexity estimates for the kernel-independent FFM}\label{subsec:complexityGeneric}

We prove here the storage and computational complexity for the kernel-independent version. Since the FFM is a descent only algorithm, we can discard data stored at the parent level in the tree. This minimal storage requirement leads to the following proposition.
\begin{proposition}\label{thm:complexityGeneric}
    The storage complexity of the FFM is $\mc O(N)$ and the computational complexity is $\mc O(N\cdot\log(N))$.
\end{proposition}
\begin{proof}
    Before we prove the estimates, we make some preliminary remarks. We first emphasize that the interpolation step for each of the {\em source} set, resp. {\em target}, is performed only once for all the corresponding subsets. Second, we drive the attention to the fact that, while there should be many interpolations to compute, most of the transfers are the same. On Figure \ref{img:transferExample}, the two transfers (arrows) represented between the boxes from the {\em source} tree (full lines) and the {\em target} tree (dotted) have the same transfer matrix which is computed only once.
    \begin{figure}[!htbp]
        \centering
        \includegraphics[height=50mm]{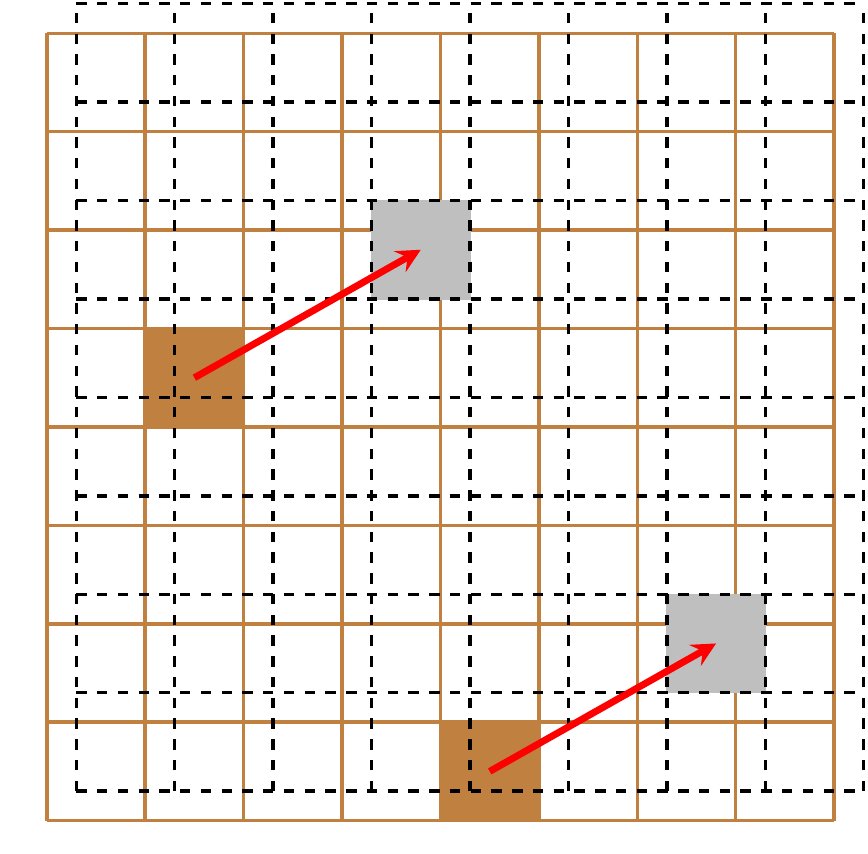}
        \caption{Example of two equivalent transfer steps.}
        \label{img:transferExample}
    \end{figure}
    In fact, the amount of different translations is uniquely determined by the initial configuration of the root bounding boxes as illustrated in 2-dimension on Figure \ref{img:quadtreeConfig}. The left picture corresponds to two overlapping quadtrees. If we enumerate all the possible low-rank interactions for the filled box at the depth 2, we find $27$. The filled bounding box interacts at most with the three outer layers minus the closest layer of boxes meaning there are {\em at most} $7^2 - 3^2 = 40$ possibilities per box at all levels of refinement. All the other boxes are children of boxes involved in low-rank interactions at the previous level of refinement. On the right configuration, the trees are shifted and the amount of possible interactions is now $48$. This number corresponds to the {\em total number of transfer matrices} which needs to be computed at the current depth $l$; it does {\em not} depend on $l$.\par
    \begin{figure}[!htbp]
        \centering
        \includegraphics[height=60mm]{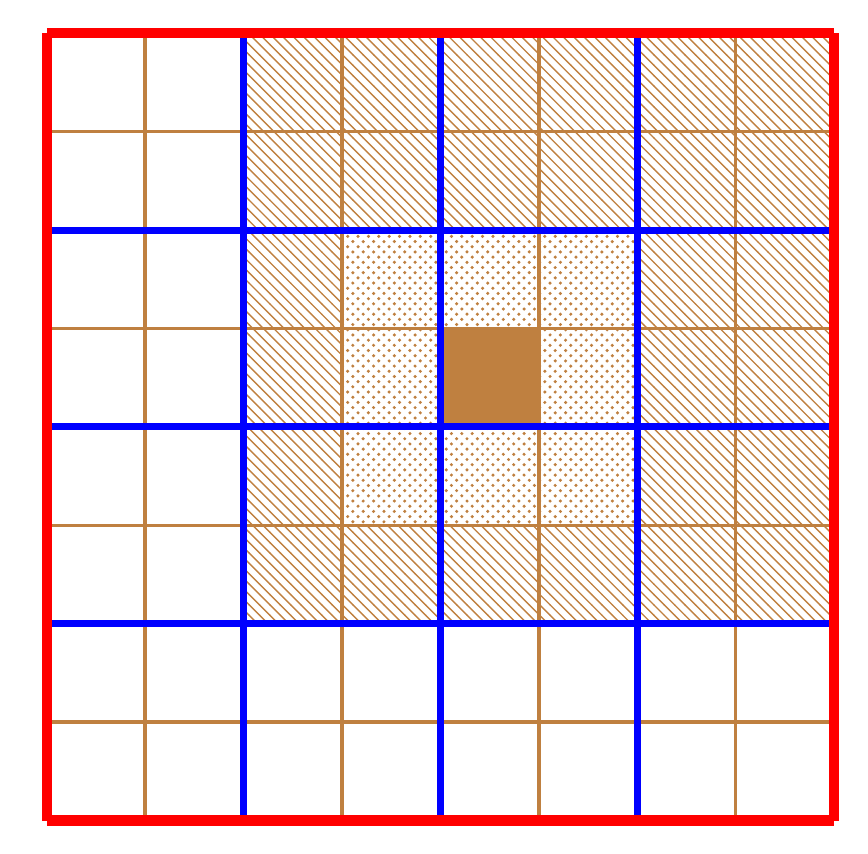}
        \includegraphics[height=60mm]{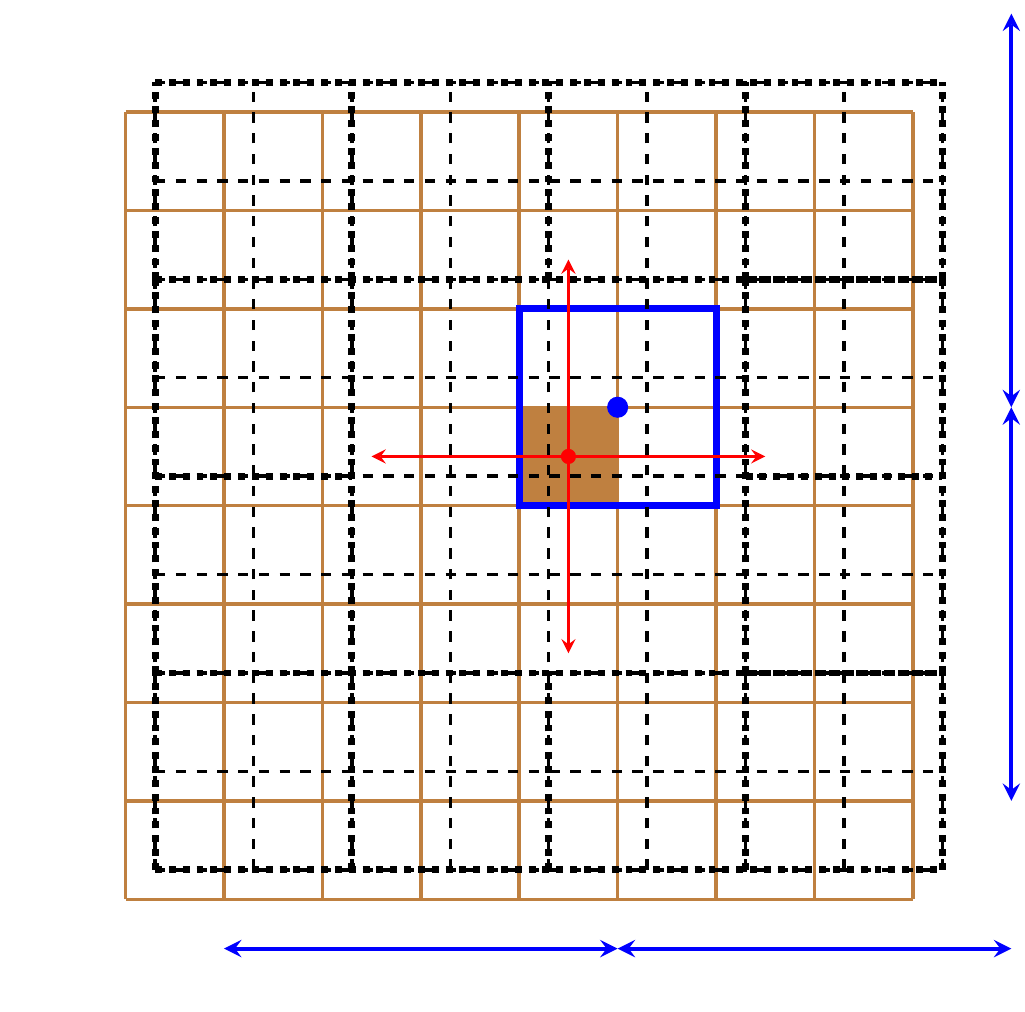}
        \caption{Space configuration of two three-leveled quadtrees in 2-dimension. Left: the quadtrees overlap. Right: the quadtree are shifted. Very thick line: level-0, thick line: level-1, thin line: level-2. Fully colored cell: bounding box for which we want to determine the low-rank interactions.}
        \label{img:quadtreeConfig}
    \end{figure}
    
    \begin{enumerate}
        \item Let $r_1$ the number of control nodes for the Lagrange interpolation in {\em each direction} of space, then $r=r_1^3$ is bounded above by a value independent on $N$. Let $n_Y = N_Y/8^l$ and $n_X = N_X/8^l$, then the storage requirement for the first interpolation is $\mc O(n_Y\cdot r) = \mc O(n_Y)$ per interpolation matrix. Since there are as many matrices as there are boxes in the tree at level $l$, the storage requirement is then $8^l\cdot \mc O(n_Y) = \mc O(N_Y)$. The interpolation consists in $8^l$ matrix-vector products, each of them of size $r\times n_Y$. The computational complexity of the complete Lagrange interpolation of the {\em source} set for any $l$ is therefore $\mc O(N_Y)$.
        \item Let $N_{\mbf T}$ be the maximum number of unique transfers (see again Figure \ref{img:quadtreeConfig}), the storage and the computation of these matrices are each $\mc O(1)$ as the number and the dimensions of the transfer matrices do not depend on $N_Y$ nor $N_X$. Each interpolated {\em source} vector is then multiplied {\em most} $N_{\mbf T}$-times by a transfer matrix whose size is independent on $N_X$ or $N_Y$. Consequently, this transfer step has $\mc O(N)$ complexity for the storage and the computational complexity.
        \item Finally, the last interpolation is performed in two steps:
        \begin{enumerate}
            \item The construction of the interpolation matrices for the {\em target} set has the same complexity as for the {\em source} set, i.e. it is $\mc O(N_X)$.
            \item The transfer to the {\em target} set of interpolations consist in computing the "interpolation matrix"-"transferred vector" for each of the $N_{\mbf T}\cdot 8^l$ "transferred vector". This step should be performed at most $N_{\mbf T}$ times for each of the {\em target} i.e. $N_{\mbf T}\cdot 8^l$-times.
        \end{enumerate}
        Therefore, the second interpolation step has the same complexity as the step 1, i.e. it is $\mc O(N_X)$ in storage and computational complexity.
    \end{enumerate}
    To these costs, we must add the worst cost of the construction of the octrees which is $\mc O(N\cdot\log(N))$ and of the non-compressible interactions which is linear because the maximum number of close interactions for one node is bounded above by the maximum number of nodes allowed in the leaves of the tree.\par
    We conclude that the storage requirement for the FFM is $\mc O(\max(N_X, N_Y))$. Regarding the computational requirement, the worst case consist in performing the interpolation $l^{\tn{max}}$-times. We conclude that the computational cost for the complete FFM product is $\mc O(N\cdot\log(N))$.
\end{proof}

\begin{remark}
    These complexity estimates do not depend on wether the nodes are evenly scattered in a volume (octree-based partitioning) or on a surface (equivalent quadtree-based partitioning).
\end{remark}

\subsection{Complexity estimates for the oscillating kernels}\label{subsec:complexityOscillating}

In the previous subsection \ref{subsec:complexityGeneric}, we proved very general estimates. When dealing with oscillating kernels, one must introduce the notion of wavelength $\lambda$ and of "discretization" per wavelength. Let $\Delta x$ be the average distance between two nodes, it is generally chosen such that 

\begin{align}
k\cdot\Delta x = C, \hspace{10mm} C \approx 1,
\end{align}
where we recall that $k=2\pi/\lambda$ is the wavenumber. In other words, there are approximately six nodes per wavelength. This is fundamentally different from the generic case as the number of nodes in a given volume depends {\em explicitly} on the wavenumber . We assume in the following that $k$ is accordingly adjusted as a function of $N$.\par
We suppose again that the nodes are scattered evenly in a volume. The special case of surfaces is tackled at the end of this subsection.
\begin{proposition}\label{thm:complexityOscillating}
    The storage complexity for the oscillating FFM is $\mc O(N)$ and the computational complexity is $\mc O(N\cdot\log^2(N))$.
\end{proposition}
\begin{proof}
    The proof is very similar to the proof of {\bf Proposition} \ref{thm:complexityGeneric}. We first remark that the NUFFT is based on the Fast Fourier Transform which has a $\mc O(\max(n_Q, n_Y))$ storage requirement and a $O(\max(n_Q, n_Y))\cdot\log(O(\max(n_Q, n_Y)))$ computational complexity. These estimates remain valid for the NUFFT but with a bigger multiplicative constant, see \cite{greengardNufft}. The main difficulty is to estimate the dependency of $n_Q$ to $N$. To that purpose, we assume that for some $l$ we have $n_Q \ll N$. In our choice of spherical quadrature, we have

    \begin{align}
        n_Q = 2\cdot L\cdot (L+1)
    \end{align}
    where $L$ is given by (\ref{eq:gegenbauerRank}). By substituting $d_l=d_0/2^l$ and expanding $n_Q$, we obtain 

    \begin{align}\label{eq:polynomnQk}
        n_Q = A\left(\dfrac{k}{2^l}\right)^2 + B\left(\dfrac{k}{2^l}\right) + C
    \end{align}
    where $A,B,C$ are constants. Following the hypothesis $k\cdot\Delta x\approx 1$, we have that 

    \begin{align}\label{eq:polynomnQN}
        N \approx(\dfrac{d_0}{\Delta x})^3\hspace{5mm}\Leftrightarrow\hspace{5mm} k\sim N^{1/3}.
    \end{align}
    Consequently,

    \begin{align}
        n_Q = A\dfrac{N^{2/3}}{4^l} + B\dfrac{N^{1/3}}{2^l} + C
    \end{align}
    meaning that $n_Q = \mc O(N^{2/3})$. In the particular case of surface distributions of nodes, we have $k\sim N^{1/2}$ and the polynomial (\ref{eq:polynomnQN}) becomes
    
    \begin{align}
        n_Q = A\dfrac{N}{4^l} + B\dfrac{N^{1/2}}{2^l} + C.
    \end{align}
    This means that for surface distributions of nodes, $n_Q = \mc O(N)$. Since the coefficients $A,B,C$ are all positive, we conclude that we always have $n_Q \ll N$.\par 
    
    We detail the complexity of each step below:
    \begin{enumerate}
        \item The NUFFT has the same complexity as the classical FFT. Therefore, the storage requirement for {\em one} NUFFT is $\mc O(n_y)$ meaning that the total storage requirement is $\mc O(N_Y)$. The total number of NUFFTs is $8^l$ and each has the computational complexity $\mc O(n_Y\cdot\log(n_Y))$. Since $n_Q < \max(n_Y, n_X)$, we conclude that the storage complexity of the first step is $\mc O(N)$ and that the computational complexity is $\mc O(N\cdot\log(N))$.
        \item The second step is a simple multiplication in the Fourier domain which is performed $N_{\mbf T}\cdot 8^l$-times on vectors of length $n_Q$, see the proof of {\bf Proposition} \ref{thm:complexityGeneric}. By remarking that $8^l\cdot n_Q = \mc O(N)$, we conclude that this step has a linear storage and computational complexity.
        \item The last step is the backward step of the first one. Consequently, we obtain exactly the same complexities.
    \end{enumerate}
    By including the complexity of the tree and the computation of the non-compressible interactions (see the end of the proof of {\bf Proposition} \ref{thm:complexityGeneric}), we conclude that the worst storage requirement is still $\mc O(N)$ and that the worst computational complexity is $\mc O(N\cdot\log^2(N))$.
\end{proof}

\begin{remark}
In this section, we proved that the storage requirement of the FFM is always linear while the computational complexity is $\mc O(N\cdot\log(N))$ in the general case (but with eventually big multiplicative constants) and $\mc N\cdot\log(N)^2)$ for the particular case of oscillating kernels. These estimates are not worse than the existing compression methods and they are illustrated in the next section.
\end{remark}
\section{Numerical Examples}

In this section, we give examples of application of the FFM. Our implementation is written in the \tsc{MATLAB} language. First we illustrate the estimates proven in the previous section \ref{sec:complexity}. In a second time, we show how one can use the FFM to solve Boundary Integral Equations iteratively and we solve two publicly available benchmarks in acoustic and electromagnetic scattering.

\subsection{Scalability of the FFM}\label{subsec:numericScalability}
The scalability of the FFM is illustrated by computing the convolution product where the kernel is the Laplace Green kernel given below,

\begin{align}\label{eq:laplaceGreenKernel}
    G(x,y) = \dfrac{1}{4\pi}\dfrac{1}{|x-y|}.
\end{align}
Physically, it amounts to compute the gravitational potential generated by masses located at the {\em source} nodes at the {\em target} nodes. To that purpose, we pick randomly $N$ {\em source} nodes and $N$ {\em target} nodes on the unit sphere $\mbb S^2$ centered at the origin as illustrated on Figure \ref{img:validationSetting}. 
\begin{figure}[!htbp]
    \centering
    \includegraphics[height=80mm]{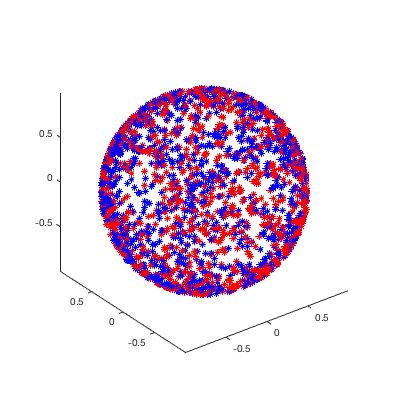}
    \caption{{\em Source} (red) and {\em target} nodes (blue)}
    \label{img:validationSetting}
\end{figure}
Then, we simply compute the convolution product (\ref{eq:convolution_matrix}) with the FFM and we measure the memory requirement and the computation time. The computation\footnote{The convolution product between two sets of nodes may be computed using the \ttt{ffmProduct()} function, available within the open-source Gypsilab framework \cite{gypsiHub} in the \ttt{./openFfm} directory. An example is provided by running the \ttt{nrtFfmBuilder.m} script.} was performed on a computer with 12 cores at 2.9 GHz, 256 GBytes of RAM and Matlab R2017a using single precision. The results are gathered in Table \ref{tab:validationLaplace} for a prescribed accuracy $\varepsilon=10^{-3}$ on the matrix-vector product. The linear storage requirement is confirmed. We observe furthermore that the computational scalability is close to linear. More importantly we are able to achieve a matrix-vector product with {\em one billion of nodes} in each of the {\em source} and {\em target} set in less than four hours. On the required 100 GBytes, approximately 40 GBytes are required for the storage of the coordinates of the nodes, the input vector and the output vector. This implies that the multiplicative constant in the $\mc O(N)$ estimate is almost $2$.\par 
\renewcommand{\arraystretch}{1.5}
\begin{table}[!htbp]
    \centering
    \begin{tabular}{|c|c|c|c|c|}
        \hline
        N & Time 1 core (s) & Time 12 cores (s) & Error & Memory \\
        \hline
        $10^4$ & 2.04 & 9.08 & $8.03\cdot 10^{-5}$ & 1 MB \\
        \hline
        $10^5$ & 9.30 & 17.1 & $1.34\cdot 10^{-4}$ & 10 MB \\
        \hline
        $10^6$ & 87.8 & 33.4 & $1.35\cdot 10^{-4}$ & 100 MB \\
        \hline
        $10^7$ & 1063 & 169 & $1.98\cdot 10^{-4}$ & 1GB \\
        \hline
        $10^8$ & -- & 1499 & $1.81\cdot 10^{-4}$ & 10 GB \\
        \hline
        $10^9$ & -- & 11340 & $3.11\cdot 10^{-4}$ & 100 GB \\
        \hline
    \end{tabular}
    \caption{Summary of the computation time and memory requirement -- Laplace kernel.}
    \label{tab:validationLaplace}
\end{table}
\renewcommand{\arraystretch}{1}
The same experiment is repeated for the Helmholtz Green kernel. The results are given in Table \ref{tab:validationHelmholtz} with the corresponding maximum $k\cdot r$ value.
\renewcommand{\arraystretch}{1.5}
\begin{table}[!htbp]
    \centering
    \begin{tabular}{|c|c|c|c|c|c|}
        \hline
        N & f (Hz) & $k\cdot r_{\tn{max}}$ & Time 1 core (s) & Time 12 cores (s)  & Error \\
        \hline
        $10^4$ & 541 & 20 & 1.65 & 8.81 & $2.98\cdot 10^{-4}$ \\
        \hline
        $10^5$ & 1893 & 70 & 16.2 & 16.3 & $1.69\cdot 10^{-4}$ \\
        \hline
        $10^6$ & 5411 & 200 & 143 & 48.2 & $2.77\cdot 10^{-4}$ \\
        \hline
        $10^7$ & 16234 & 600 & 1557 & 350 & $2.91\cdot 10^{-4}$ \\
        \hline
        $10^8$ & 54113 & 2000 & -- & 8246 & $3.52\cdot 10^{-4}$ \\
        \hline
    \end{tabular}
    \caption{Summary of the computation times -- Helmholtz kernel.}
    \label{tab:validationHelmholtz}
\end{table}
\renewcommand{\arraystretch}{1}

\subsection{Boundary Integral Equations and the FFM}\label{subsec:numericBEM}

Boundary Integral Equations can be obtained from certain equations describing, for example, physical phenomenons like the propagation of an acoustic or electromagnetic wave in a homogeneous medium. We refer to \cite{nedelec} starting p. 110, or \cite{coltonKress} starting p. 66, for more details on how they are obtained. In order to explain how the FFM is used to solve such equations, we consider the scattering of an acoustic wave propagating in an infinite medium by a scatterer with boundary $\Gamma$ on which we apply a Dirichlet boundary condition. This problem can be tackled by solving the following Boundary Integral Equation

\begin{align}\label{eq:singleLayerEquation}
    \int_{\Gamma}{G(x,y)\,\lambda(y)\,d\gamma_y} = -u^{\tn i}(x),\hspace{5mm}x\in\Gamma
\end{align}
where $\lambda$ is some unknown, $G(x,y)$ is the Helmholtz Green kernel (see (\ref{eq:helmholtzKernel})) and $u^{\tn i}$ is the incident wave. This equation may be solved using different method. Here we present shortly the Boundary Element Method based on a Galerkin formulation. We introduce a test function $\lambda^{\star}$ to obtain the Galerkin formulation of eq. (\ref{eq:singleLayerEquation})

\begin{align}\label{eq:singleLayerBEM}
    \int_{\Gamma\times\Gamma}{\lambda^{\star}(x)\,G(x,y)\,\lambda(y)\,d\gamma_y\,d\gamma_x} = -\int_{\Gamma}{u^{\tn i}(x)\,\lambda^{\star}(x)}\hspace{5mm}\tn{for all}\hspace{1mm}\lambda^{\star}.
\end{align}
We further introduce the discrete approximation spaces $(\lambda_j)_{j\in[\![1,N]\!]}$ and $(\lambda^{\star}_i)_{i\in[\![1,N]\!]}$ such that

\begin{align}
    \lambda(y) &= \sum_{j=1}^N{u_j\cdot\lambda_j(y)}, \\
    \lambda^{\star}(y) &= \sum_{i=1}^N{v_i\cdot\lambda^{\star}_i(x)}.
\end{align}
There are multiple ways to deal with this singular integral. Here we integrate the double integral using a Gauss-Legendre quadrature. The resulting inaccurate integration of the singularity is tackled later. Let $\{\omega_{g,k},x_{g,k}\}_{k\in[\![1,n_g]\!]}$ and $\{\omega_{g,k},y_{g,k}\}_{k\in[\![1,n_g]\!]}$ be the weight and nodes of quadrature, the eq. (\ref{eq:singleLayerBEM}) now reads

\begin{align}\label{eq:singleLayerQuad}
    \int_{\Gamma\times\Gamma}{\lambda^{\star}(x)\,G(x,y)\,\lambda(y)\,d\gamma_y\,d\gamma_x} \approx \sum_{i=1}^N{v_i\sum_{k=1}^{n_g}{\lambda_i^{\star}(x_{g,k})\,\omega_{g,k}\,\sum_{l=1}^{n_g}{G(x_{g,k}, y_{g,l})\,\sum_{j=1}^N{\omega_{g,l}\,\lambda_j(y_{g,l})\, u_j}}}}.
\end{align}
Therefore, eq. (\ref{eq:singleLayerEquation}) is rewritten as linear system of equations,

\begin{align}
    \mbf S\cdot \lambda = \mbf U^{\tn i}.
\end{align}
The Galerkin matrix $\mbf S$ can be recast as the product of three matrices such that

\begin{align}
    \mbf S = \left(\Lambda^{\star}\right)^T\cdot\left(\mbf G\cdot\Lambda\right)
\end{align}
where $\mbf L$ is the matrix "transporting" the basis functions to the quadrature nodes, $(\Lambda^{\star})^T$ is the matrix "quadrature-to-test-functions" and $\mbf G$ is the matrix such that $\mbf G_{ij} = G(x_{g,i}, y_{g,j})$. While the matrices $\Lambda$ and $\Lambda^{\star}$ are sparse and can be stored with linear complexity, $\mbf G$ is full. In the process of an iterative inversion algorithm such as GMRES, see \cite{saadSchultz}, one or more matrix-vector products are required,
\begin{enumerate}
    \item $\tilde{\mbf u} = \Lambda\cdot \mbf u$. This product has linear complexity.
    \item $\tilde{\mbf v} = \mbf G\cdot \tilde{\mbf u}$. This product is compressed using the FFM.
    \item $\mbf v = \left(\Lambda^{\star}\right)^T\cdot\tilde{\mbf v}$. This product also has linear complexity.
\end{enumerate}
In general, $n_g$ is closely related to the number of elements in the discretization of $\Gamma$. Assuming for example that there are three quadrature nodes per element, then $n_g = 3\cdot\left(\tn{number of elements}\right)$ meaning that the size of $\mbf G$ may be in fact much higher than the actual size of the linear system. The singular integral is computed independently using a semi-analytical method. It takes the form of an additional sparse matrix which "removes" the singularity integrated numerically in $\mbf G$ and adds the "exact" integration of the kernel.\par
Our FFM library is interfaced with the \tsc{Gypsilab} software. \tsc{Gypsilab} is an open-source (GPL3.0) Finite Element framework entirely written in the \tsc{Matlab} language aiming at assembling easily the matrices related to the variational formulations arising in the Finite Element Method or in the Boundary Element Method. Among other things, it features a complete $\mc H$-matrix algebra (sum, product, LU decomposition, \ldots) compatible with the native matrix types of \tsc{Matlab}. For more details on the capabilities of \tsc{Gypsilab} we refer to \cite{gypsilab,gypsiHub}. In the context of this paper, we use it to manage the computation of the matrices $\Lambda$, $\Lambda^{\star}$ and the right-hand-side $\mbf U^{\tn i}$ \footnote{The high-level interface with the FFM is available as an other overloaded \texttt{integral()} function within Gypsilab. An example is provided by running the \ttt{nrtFfmBuilderFem.m} script.}.\par 
We present here two examples of application. The first one corresponds to the scattering of an underwater acoustic wave by a submarine and the second one is the scattering of an electromagnetic wave by a perfectly electric conductor rocket launcher.

\subsubsection{Acoustic scattering by a submarine}\label{subsubsec:betssi}
We solve the acoustic scattering by a submarine with Neumann Boundary condition. This example is based on the BeTSSi benchmark \cite{betssi}. The mesh is provided by ESI Group and it is remeshed using the open-source Mmg Platform \cite{mmg}. We could solve this problem using the following equation 

\begin{align}\label{eq:hyperSingular}
    \int_{\Gamma}{\dfrac{\partial^2 G}{\partial n_x\partial n_y}(x,y)\,\mu(y)\,d\gamma_y} = \dfrac{\partial u^{\tn i}}{\partial n_x}
\end{align}
whose variational formulation is

\begin{multline}
    k^2\int_{\Gamma\times\Gamma}{\left(\mu^{\star}(x)\cdot G(x,y)\cdot \mu(y)\cdot (n_x\cdot n_y)\right)\,d\gamma_x\,d\gamma_y} - \int_{\Gamma\times\Gamma}{\left(\tn{rot}_{\Gamma}\mu^{\star}(x)\cdot G(x,y)\cdot\tn{rot}_{\Gamma}\mu(y)\right)\,d\gamma_x\,d\gamma_y} = \\
    \int_{\Gamma}{\mu^{\star}(x)\cdot\dfrac{\partial u^{\tn i}}{\partial n_x}(x)\,d\gamma_x}
\end{multline}
where $\mu$ is the unknown, $\mu^{\star}$ is the test function, $k$ the wavenumber, $n_x$ is the outbound normal vector at position $x$, and $\partial/\partial n_x$ is the normal derivative with respect to the variable $x$. Eq. (\ref{eq:hyperSingular}) is very ill-conditioned and is also ill-posed for some frequencies. The integral operator is called the {\em hypersingular} boundary integral operator. The scattering problem could also be solved using another boundary integral equation

\begin{align}\label{eq:doubleLayer}
    -\dfrac{\mu(x)}{2} + \int_{\Gamma}{\dfrac{\partial G}{\partial n_x}(x,y)\,\mu(y)\,d\gamma_y} = -u^{\tn i}(x)
\end{align}
which is better conditioned but which is also ill-posed for some frequencies and less accurate in practice. The boundary integral operator in eq. (\ref{eq:doubleLayer}) is the {\em adjoint of the double layer operator}. To circumvent these issues, we use a linear combination of the integral operators involved in eq. (\ref{eq:hyperSingular}) and (\ref{eq:doubleLayer}) called the {\em Brakhage-Werner} formulation, see \cite{kress}. This formulation is better conditioned and always well-posed. It is important to note that each iteration in the GMRES algorithm requires in fact $6 + 3 = 9$ FFM-products {\em in our implementation}: three for the part with the scalar product of the normal vectors, three for the part with the scalar product of the surface rotational, and three for the adjoint of the double layer operator.\par 
The submarine is $60$ m long and the frequency is set at $6.5$ kHz. Assuming, the celerity of sound in water is $1500$ m.s$^{-1}$, the wavelength is $0.231$ m meaning that there are approximately $260$ wavelengths along the submarine, or differently written we have $k\cdot r_{\tn{max}} \approx 1687$. The mesh features $12.8\cdot 10^6$ triangles. Since the numerical integration is performed using a quadrature rule with 3 nodes per triangle, the size of one FFM-product is $38.4\cdot 10^6 \times 38.4\cdot 10^6$. The problem is discretized with $\mc P^1$-elements implying that the total amount of (nodal) unknowns is $6.4\cdot 10^6$. It is solved using a preconditioned GMRES algorithm without restart on a 32 cores server at 3.0 GHz with 512 GBytes RAM. The tolerance for both the GMRES and the FFM product is set to $\varepsilon = 10^{-3}$. Convergence is achieved in 12 iterations and 50000 seconds ($\approx$ 13 hours and 50 minutes) including the assembling of the preconditioner and the regularization matrix. Each iteration requires approximately $3570$ seconds. The radiated field on the surface is represented on Figure \ref{img:betssi}. This computation is {\em also} performed using the FFM. The memory peak is measured at approximately 200 GBytes when assembling the preconditioner and the regularization matrix.
\begin{figure}[!htbp]
    \centering
    \includegraphics[width=0.8\linewidth]{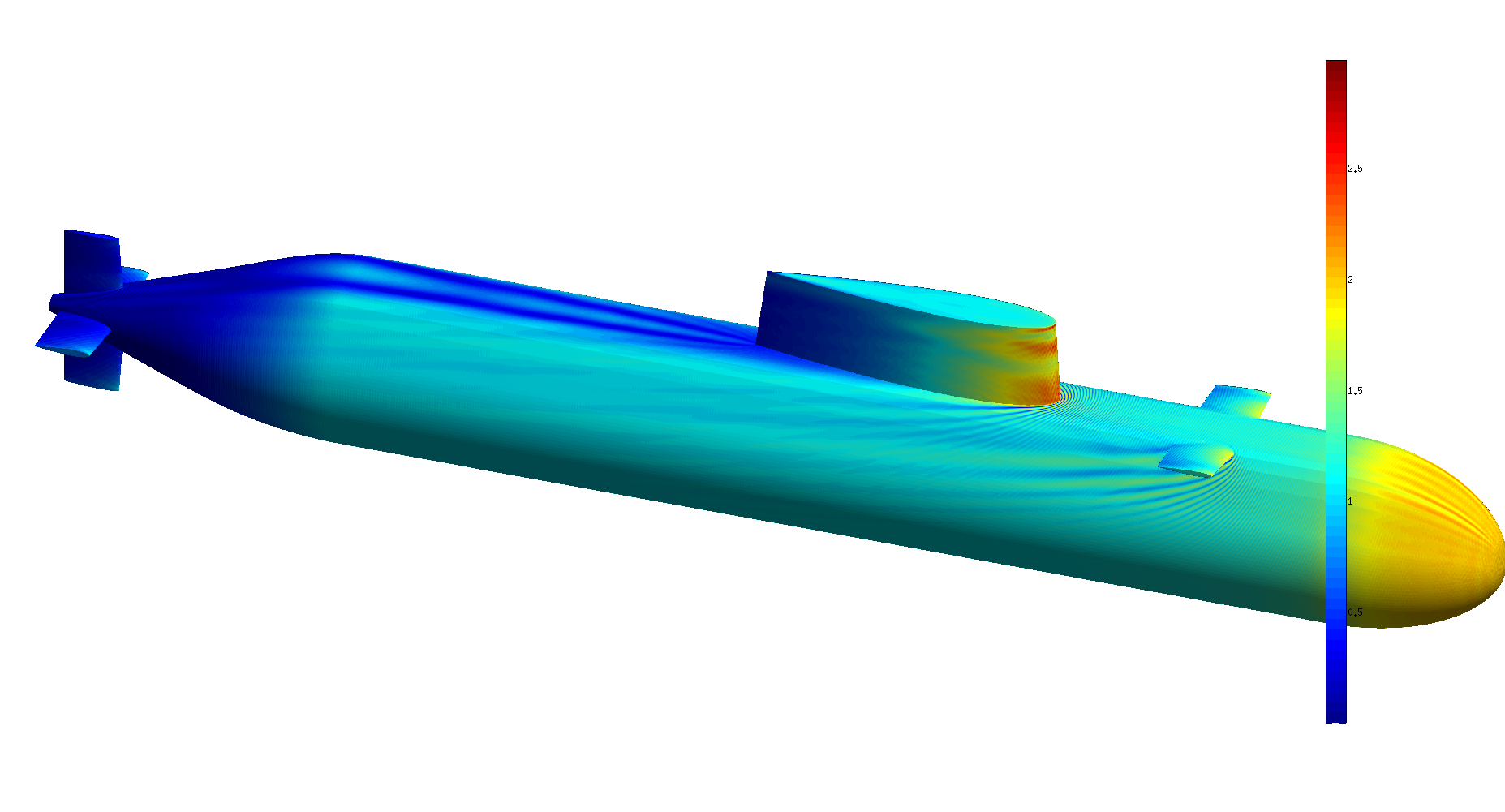}\\
    \includegraphics[width=0.8\linewidth]{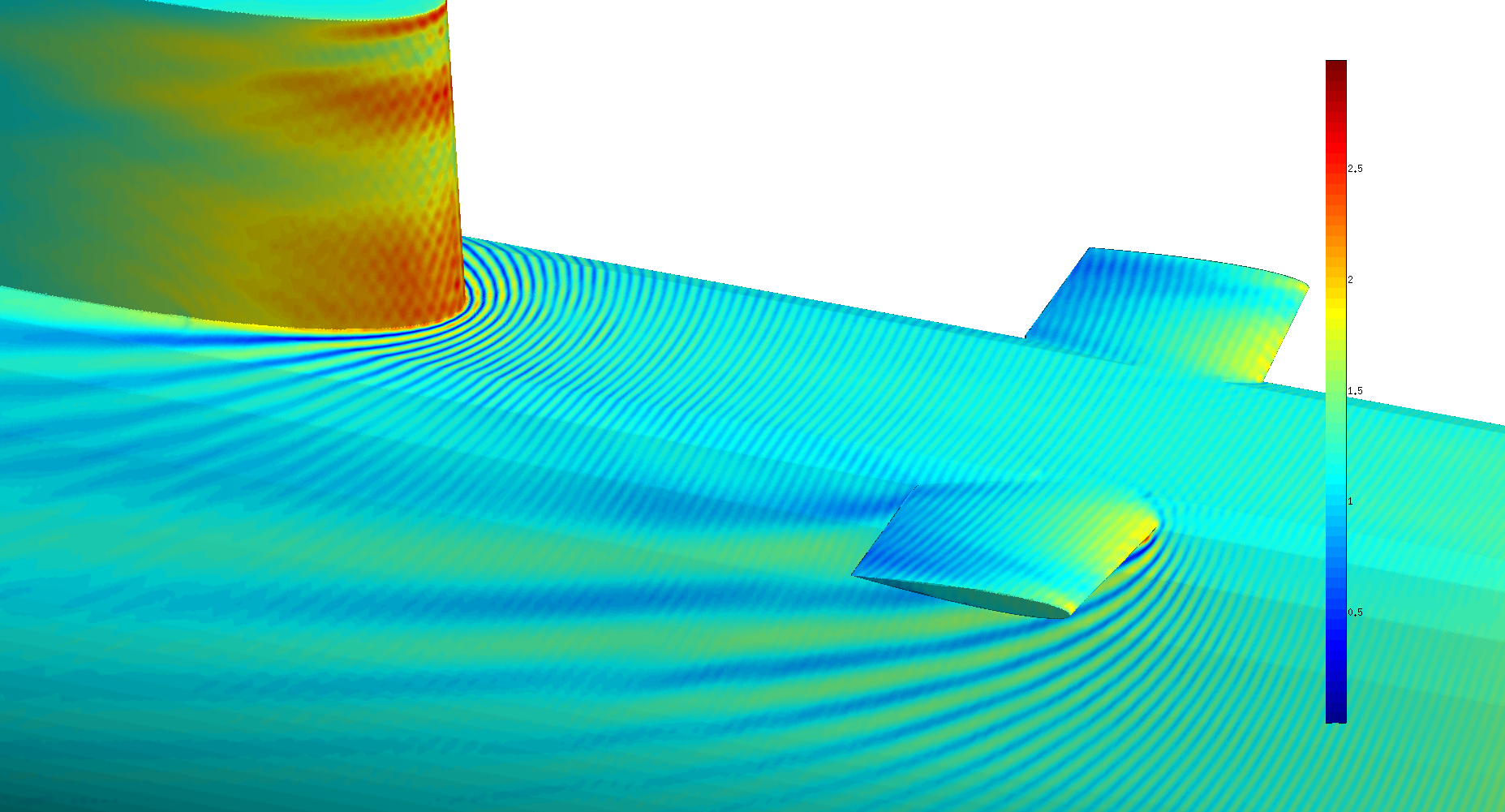}\\
    \caption{Absolute value of $\mu$ on the surface of the submarine}
    \label{img:betssi}
\end{figure}

\subsubsection{Perfectly electric rocket launcher}\label{subsubsec:launcher}
This example is a slightly modified version of a test case extracted from the Workshop EM-ISAE 2018, see \cite{workshop2018}. We study the scattering of an electromagnetic plane wave by a perfectly electric launcher. This problem is solved the Combined Field Integral Equation (CFIE) which is a linear combination of the Electric Field Integral Equation (EFIE) and the Magnetic Field Integral Equation (MFIE). For the construction of these equations, we refer once again to \cite{nedelec} starting p. 234, or \cite{coltonKress} starting p. 108. The EFIE reads

\begin{align}\label{eq:EFIE}
    \dfrac{1}{k^2}\nabla_{\Gamma}\int_{\Gamma}{G(x,y)\, \nabla_{\Gamma}\cdot\mbf J(y)\,d\gamma_y} + \left(\int_{\Gamma}{G(x,y)\mbf J(y)\,d\gamma_y}\right)_T = -\dfrac{(\mbf E^{\tn i})_T}{ikZ}
\end{align}
where $\mbf J$ is the tangential trace of the magnetic field, $\mbf E^{\tn i}$ is the incident electromagnetic wave, $Z$ is an impedance, and $(.)_T$ is the {\em tangential trace} operator. The MFIE reads

\begin{align}\label{eq:MFIE}
    \dfrac{\left(\mbf J\times n_x\right)(x)}{2} + n_x\times\int_{\Gamma}{\nabla_y G(x,y)\times\mbf J(y)\,d\gamma_y} = n_x\times \mbf H^{\tn i}(x)
\end{align}
where $\mbf H^{\tn i}$ is the incident {\em magnetic} field. The CFIE then reads

\begin{align}\label{eq:CFIE}
    \tn{CFIE} = \alpha\cdot\tn{EFIE} + (1-\alpha)\cdot Z\cdot\tn{MFIE},\hspace{10mm}\alpha\in]0,1[.
\end{align}
We use the classical Raviart-Thomas finite element space of order 0 (RT$_0$). The launcher is 60 m long and 12 m in diameter (including the boosters). The electromagnetic wave propagates at $2$ GHz meaning that the wavelength is 0.15 m, assuming a celerity of light equals to $3\cdot 10^8$ m.s$^{-1}$. Therefore, there are $400$ wavelength along the launcher. The total number of unknowns is $60\cdot 10^6$ for approximately $30\cdot 10^6$ triangles. Consequently, the size of a FFM product is $90\cdot 10^6\times 90\cdot10^6$. We use the same server as in subsubsection \ref{subsubsec:betssi}. One GMRES matrix-vector product involves now $10$ ($4$ for the EFIE, $6$ for the MFIE) FFM products and lasts approximately $4$ hours. The real part of the solution is represented on Figure \ref{img:launcher}.
\begin{figure}[!htbp]
    \centering
    \includegraphics[width=0.8\linewidth]{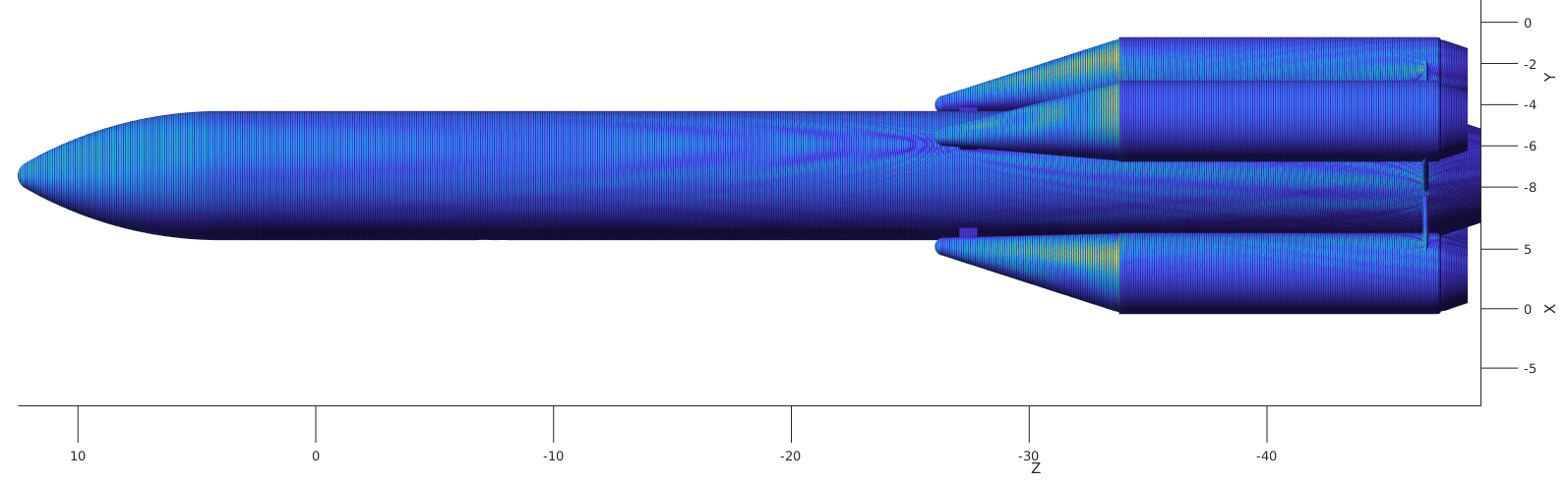}\\
    \includegraphics[width=0.8\linewidth]{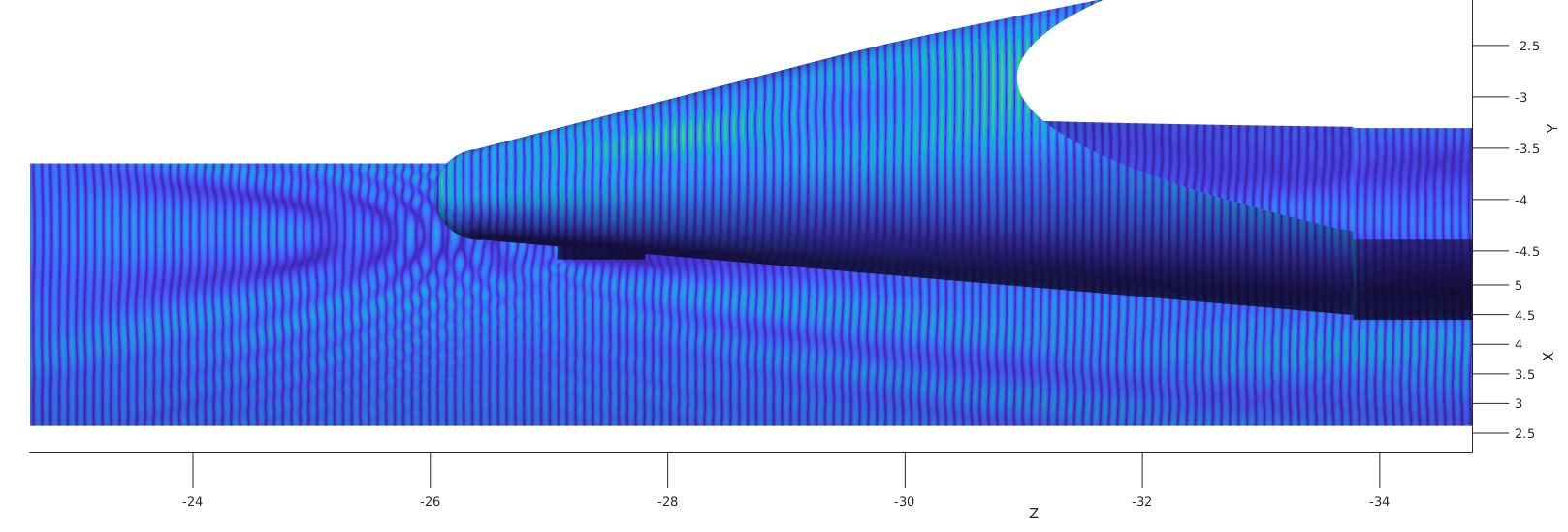}
    \caption{Real part of $\mbf J$ on the whole launcher (top) and detail (bottom).}
    \label{img:launcher}
\end{figure}

\subsection{The FFM compared to the $\mc H$-matrices}
We want to stress that for lower problems sizes ($N\sim 10^6$), the FFM performs, in general, worse than other methods like the $\mc H$-matrices or the FMM. Indeed, everything is computed anew at every call to the method in order to save memory. When compared to the $\mc H$-matrices, the assembling of the $\mc H$-matrix costs more than a few FFM products but the $\mc H$-matrix-vector product is then almost free. This is illustrated in Table \ref{tab:ffmAndHmat} for the Laplace Green kernel. The computation is performed on a laptop and a single core at 2.6 GHz is used. The total available RAM is 16 GBytes. We use the $\mc H$-matrix library featured in Gypsilab. 
However, the FFM enables the computation of {\em massive} convolution products as illustrated in subsections \ref{subsec:numericScalability} and \ref{subsec:numericBEM}, even for oscillating kernels, on {\em small} servers, thus avoiding the use of huge infrastructures.
\renewcommand{\arraystretch}{1.5}
\begin{table}[!htbp]
    \centering
    \begin{tabular}{|c|c|c|c|}
        \hline
        N & Time FFM product (s) & Time $\mc H$-matrix assembly (s) & Time $\mc H$-matrix product (s) \\
        \hline\hline
        $10^3$ & 0.97 & 0.68 & 0.051 \\
        \hline
        $10^4$ & 2.16 & 5.9 & 0.13 \\
        \hline
        $10^5$ & 13.1 & 54.4 & 1.09 \\
        \hline
        $10^6$ & 103.5 & 733 (swapped)  & 122 (swapped)  \\
        \hline
    \end{tabular}
    \caption{Time comparison between the FFM and $\mc H$-matrices for the Laplace kernel -- Laptop single core @2.6 GHz and 16 GBytes RAM}
    \label{tab:ffmAndHmat}
\end{table}
\renewcommand{\arraystretch}{1}

\section{Conclusion}
We have proposed a powerful and scalable alternative to the existing compression methods when dealing with convolution featuring millions, or eventually billions, of nodes. The two main ingredients of the FFM are: a descent-only tree traversal, and cubic bounding boxes with the same edge size for the two octrees. It enables a linear storage complexity and a quasi-linear computational complexity which are numerically highlighted. The FFM is kernel-independent but its versatility enables optimization for strongly oscillating convolution kernels. Moreover, the implementation effort is small as only a few hundred Matlab lines are required. From a practical point of view, the FFM gives the possibility to compute realistic problems to users who do not have access to a computer cluster. The code is now freely available in the git repository of \tsc{Gypsilab} at \cite{gypsiHub} under the GPL 3.0 license.

\section*{Acknowledgements} 
This work is funded by the {\em Direction G\'en\'erale de l'Armement}. We would like to thank Fran\c cois Alouges for his invaluable help in the development of this work and Leslie Greengard for providing the NUFFT code used for the computation of the numerical results.

\bibliographystyle{plain}
\bibliography{biblio}
\end{document}